\documentclass[12pt]{article}
\usepackage{graphicx,psfrag,epsf}
\usepackage{enumerate}
\usepackage{amsmath,amsthm,amsfonts,epsfig,amssymb,natbib,eucal,eufrak}
\usepackage{booktabs}
\usepackage{multirow}
\usepackage{siunitx}
\usepackage{qtree}
\usepackage{hyperref}
\newcommand{\MRhref}[1]{\href{https://mathscinet.ams.org/mathscinet-getitem?mr=#1}{MR#1}}

\usepackage{float}
\restylefloat{table}
\usepackage{color}

\usepackage{pdfpages}
\usepackage{wrapfig}

\usepackage{comment}

\usepackage{tikz}

\newtheorem{thm}{Theorem}
\newtheorem{lem}{Lemma}

\newtheorem{definition}{Definition}

\newtheorem{result}{Result}

\newtheorem{cor}{Corollary}
\newtheorem{pro}{Proposition}

\newcommand{\mP}{\mathbb{P}}

\newcommand{\blind}{0}

\begin{document}

\def\spacingset#1{\renewcommand{\baselinestretch}%
{#1}\small\normalsize} \spacingset{1}

%%%%%%%%%%%%%%%%%%%%%%%%%%%%%%%%%%%%%%%%%%%%%%%%%%%%%%%%%%%%%%%%%%%%%%%%%%%%%%

\if0\blind
{
  \title{\bf Extreme Score Distributions in Countable-Outcome Round-Robin Tournaments of Equally Strong Players\\
}

  \author{
  Yaakov Malinovsky
    \thanks{email: yaakovm@umbc.edu}
   \\
    Department of Mathematics and Statistics\\ University of Maryland, Baltimore County, Baltimore, MD 21250, USA\\
}
  \maketitle
} \fi

\if1\blind
{
  \bigskip
  \bigskip
  \bigskip
  \begin{center}
    {\LARGE\bf Title}
\end{center}
  \medskip
} \fi
\begin{abstract}
We consider a general class of round-robin tournament models of equally strong players. In these models, each of the $n$ players competes against every other player exactly once. For each match between two players, the outcome is a value from a countable subset of the unit interval, and the scores of the two players in a match sum to one.
The final score of each player is defined as the sum of the scores obtained in matches against all other players. We study the distribution of extreme scores, including the maximum, second maximum, and lower-order extremes. Since the exact distribution is computationally intractable even for small values of $n$, we derive asymptotic results as the number of players $n$ tends to infinity, including limiting distributions, and rates of convergence.

\end{abstract}

\noindent%
{\it Keywords: Complete graph, negative association, order statistics, Poisson approximation, total variation distance}

\noindent%
{\it MSC2020: 62G32; 05C20; 60C05}
%\hfill {\tiny technometrics tex template (do not remove)}

\spacingset{1.42} % DON'T change the spacing!
\section{Introduction and Background}
Round-robin tournaments constitute a classical and extensively studied model in combinatorics, probability, and statistics. In such tournaments, each of the $n$ players competes against each of the others $n-1$ players, and the outcome of each match contributes to the players’ final scores. The structure and distribution of these scores have been studied from combinatorial, probabilistic, and statistical perspectives; however, the behavior of extreme values has received comparatively little attention.

In an 1895 paper, Edmund Landau proposed and investigated an alternative scoring system for chess tournaments, differing from the classical scheme, still in use today, in which the winner of a game receives $1$ point, the loser $0$ points, and, in the case of a draw, each player receives $\tfrac{1}{2}$ point \citep{L1895,L1914}.
More recently, \cite{SZ2022} observed that Landau’s idea of determining scores through relative comparisons plays a central role in the original 1998 \textit{Google} search algorithm, PageRank.

Another eminent mathematician, Ernst Zermelo, referred to Landau’s 1914 work \citep{L1914} and proposed a probabilistic model in which the probability that a player wins a game is given by the ratio of that player’s strength to the sum of the strengths of the two competing players. He estimated the model parameters using the method of maximum likelihood and applied this approach to the analysis of the 1924 New York chess tournament \citep{Z1929}. The Zermelo model was later rediscovered by \cite{BT1952} and by \cite{F1957}; see also the detailed historical discussion in \cite{DE2001}, as well as the probabilistic aspects of rating models discussed in \cite{A2017}.

Formally, in a round-robin tournament, each of the $n$ players competes against each of the other $n-1$ players exactly once. When player $i$ competes against player $j$, the reward obtained by player $i$ is a random variable $X_{ij}$. We define
\begin{equation*}
%\label{eq:score}
s_i(n)=\sum_{\substack{j=1 \\ j\neq i}}^{n} X_{ij}
\end{equation*}
to be the score of player $i$, $1 \le i \le n$, after playing all other $n-1$ players.

Throughout this paper, we distinguish between models that allow different sets of possible values for the random variables $X_{ij}$. Specifically, we consider the following cases:
\begin{enumerate}
\item
$D_k = \left\{0,\frac{1}{k}, \frac{2}{k},\ldots,1\right\}$,

%\item
%$D_{k_{ij}} = \{0,1,\ldots,k_{ij}\}$, where the set of possible values may depend on the pair $(i,j)$,

\item
$D \subset [0,1]$ is a countable set.
\end{enumerate}

In the classical round-robin tournament, $X_{ij}+X_{ji}=1$ and $X_{ij}\in D_{1}$, whereas in the chess round-robin tournament discussed above, $X_{ij}+X_{ji}=1$ and $X_{ij}\in D_{2}$. We assume that the $\binom{n}{2}$ random vectors
$
\{(X_{ij},X_{ji}) : 1 \le i < j \le n\}
$
are independent.

We refer to the vector
$$
\bigl(s_1(n), s_2(n), \ldots, s_n(n)\bigr)
$$
as the \emph{score sequence} of the tournament. The corresponding order statistics are denoted by
$$
s_{(1)}(n)\leq s_{(2)}(n)\leq \cdots \leq s_{(n)}(n).
$$

We also denote by
$$
\bigl(s_1^{\star}(n), s_2^{\star}(n), \ldots, s_n^{\star}(n)\bigr)
$$
the normalized score sequence of the tournament, where, for a given model, each score is centered by subtracting its expectation and scaled by dividing by its standard deviation. The corresponding order statistics are denoted by
$$
s_{(1)}^{\star}(n)\leq s_{(2)}^{\star}(n)\leq \cdots \leq s_{(n)}^{\star}(n).
$$

The exact distribution of score sequences in the classical round-robin tournament for $n\le 9$, and in the chess round-robin model for $n\le 6$ under the assumption of equally strong players, was determined by the prominent mathematician Percy A.~MacMahon \citep{M1923a,M1923b}. However, even for small values of $n$, determining the exact distribution is in general computationally infeasible due to the astronomical size of the sample space, which, for example, in the classical round-robin tournament equals $2^{\binom{n}{2}}$.

An excellent introduction to the combinatorial and probabilistic aspects of classical round-robin tournaments, together with references to earlier work, is provided in \cite{Moon1968}; see also the corrected 2013 version available via \href{https://www.gutenberg.org/ebooks/42833}{Project Gutenberg}.

Round-robin tournaments may be viewed as collections of paired comparisons among players. The combinatorial, probabilistic, and statistical aspects of paired-comparison models, often referred to as the method of paired comparisons, a term coined by Thurstone \citep{T1927}; see also \citep{DE2001}, have been extensively studied in the literature. Classical and modern contributions include, among many others, \citep{KBS1940, K1955, B1956, D1959, H1963, D1988, J1988, CS1998, ABKKW2006, ACKPR2017, CDL2017, BF2018, Csato2021, R2021}. Closely related problems in ranking and learning from pairwise comparisons have also attracted considerable attention; see, for example, \citep{H2019, N2023}. A comprehensive treatment of the statistical aspects of paired comparisons, together with combinatorial methods, experimental designs, and selection and ranking procedures, is provided in the monograph by Herbert A.~David \citep{D1988}.

The behavior of extreme scores in round-robin tournaments has received relatively little explicit attention in the literature, although the earliest result known to us appears already in \citep{H1963}. The investigation of such behavior is closely related to negative dependence among the scores, which we discuss next, together with relevant references, previous results, and a precise problem formulation.

We define the following notions of negative dependence. See \cite{JP1983} and the references therein for details.
Throughout this paper, the terms increasing and decreasing are understood to mean nondecreasing and nonincreasing, respectively.

\begin{definition}(\cite{JP1983}, Definition~2.3)\label{def:NOD}
The random variables $S_1,\ldots,S_n$, or the vector ${\bf S}=(S_1,\ldots,S_n)$, are said to be \textit{negatively lower orthant dependent} (NLOD) if, for all $s_1,\ldots,s_n\in\mathbb{R}$,
\begin{equation}\label{eq:L}
\mP\left(S_1 \le s_1,\ldots,S_n \le s_n\right)
\le
\mP\left(S_1 \le s_1\right)\cdots \mP\left(S_n \le s_n\right),
\end{equation}
and \textit{negatively upper orthant dependent} (NUOD) if
\begin{equation}\label{eq:U}
\mP\left(S_1 > s_1,\ldots,S_n > s_n\right)
\le
\mP\left(S_1 > s_1\right)\cdots \mP\left(S_n > s_n\right).
\end{equation}
\textit{Negative orthant dependence} (NOD) is said to hold if both \eqref{eq:L} and \eqref{eq:U} hold.
\end{definition}

The notion of negative association was studied early in \cite{AS1981}; we recall the standard definition in the form given in \cite{JP1983}.

\begin{definition}(\cite{JP1983}, Definition 2.1)\label{def:NA}
	The random variables $S_1,\ldots,S_n$ or the vector ${\bf S}=(S_1,\ldots,S_n)$ are said to be {\it negatively associated} (NA) if for every pair of disjoint subsets
	$A_1, A_2$ of $\left\{1,2,\ldots,n\right\}$,
	\begin{equation*}
%\label{eq:cov}
	\mathrm{Cov}\left(f_1(S_i, i\in A_1), f_2(S_j, j\in A_2)\right)\leq 0,
	\end{equation*}
	whenever $f_1$ and $f_2$ are real-valued functions,  increasing  in all coordinates.
\end{definition}
NA implies NOD  (see \cite{JP1983}). 

For the classical round-robin tournament model with
$p_{ij}=\mathbb{P}(X_{ij}=1)=\tfrac{1}{2}$, Huber \citep{H1963}
proved that
\begin{equation}
\label{eq:Huber}
s_{(n)}^{\star}(n)-\sqrt{2\log(n-1)} \xrightarrow{\mathbb{P}} 0
\quad \text{as } n\to\infty,
\end{equation}
where $\log(x)$ denotes the natural logarithm. One of the key elements in his proof was establishing the NLOD property~\eqref{eq:L} of the score vector
$\bigl(s_1(n), s_2(n), \ldots, s_n(n)\bigr)$ via a coupling argument.
For round-robin tournaments with repeated games, where
$X_{ij}\sim \mathrm{Bin}(r_{ij},p_{ij})$ independently for all $i<j$ and
$X_{ij}+X_{ji}=r_{ij}$, Ross \citep{R2021} established the NLOD property~\eqref{eq:L} for the score vector
$\bigl(s_1(n), s_2(n), \ldots, s_n(n)\bigr)$, which allows the study of probabilities of the form
$\mathbb{P}\!\left(s_i(n)>\max_{j\neq i}s_j(n)\right)$.
Under the model of \cite{R2021}, with integer-valued $X_{ij}$ but allowing for an arbitrary discrete distribution in place of the binomial one, \cite{MM2022} established the NLOD property~\eqref{eq:L} for the score vector
$\bigl(s_1(n), s_2(n), \ldots, s_n(n)\bigr)$ and used this property to derive a Huber-type result of the form~\eqref{eq:Huber}.
For the chess round-robin model with equally strong players,
\cite{M2021a,M2021b} obtained the asymptotic distributions of the maximal score
$s_{(n)}(n)$, the second maximal score $s_{(n-1)}(n)$, and, more generally, of all order statistics down to $s_{(1)}(n)$.
Finally, \cite{E1967} stated, without proof, that in a classical round-robin tournament with equally strong players, the probability that there is a unique winner with the maximal score tends to $1$ as $n\to\infty$. This conjecture was proved in \cite{MM2024}. More recently, \cite{AM2026} extended this result to the case in which the match outcomes $X_{ij}$ take values in a countable set $D\subset[0,1]$, and \cite{M2026} identified, as a function of $n$, the size of the sets of the largest and smallest scores that are unique with probability tending to one.

It is interesting to note that related types of problems have been extensively studied in the theory of random graphs, where the nature of dependence is, in a sense, opposite to that arising in round-robin tournaments, which exhibit negative dependence. A round-robin tournament can be represented as a complete oriented graph on $n$ vertices, where the vertices correspond to players. For each unordered pair of distinct vertices $i$ and $j$, exactly one directed edge is present: it is oriented from $i$ to $j$ if $X_{ij}=a$, and from $j$ to $i$ if $X_{ji}=1-a$.
By contrast, in the Erd\H{o}s--R\'enyi random graph model \citep{G1959, ER1959, ER1960}, there are $n$ vertices, and each unordered pair of distinct vertices $i$ and $j$ is connected by an undirected edge independently with probability $p$ (which may depend on $n$). If an edge is present, no orientation is assigned. The seminal papers \cite{ER1959,ER1961} initiated the study of the asymptotic behavior of extreme degrees (maximal and minimal) in random graphs. Subsequently, \cite{I1973} studied the asymptotic marginal distributions of the order statistics of vertex degrees in the sparse regime where $p=p(n)\to 0$ as $n\to\infty$, while Bollob\'as \citep{B1980,B1981} independently established corresponding results for fixed $p$, with further generalizations given in \citep{B2001}; see also \citep{VZ2024,M2024}.

In this paper, we extend previous results on the asymptotic distributions of the order statistics $s_{(i)}(n)$, $i=1,\ldots,n$, obtained in \citep{M2021a,M2021b} for the chess round-robin tournament model, to a more general setting in which the match outcomes $X_{ij}$ take values in a countable set $D\subset[0,1]$. In addition, we establish rates of convergence.

The remainder of the paper is organized as follows. In Section~2 we introduce the general round-robin tournament model with countable outcomes and provide additional notation and technical facts. Section~3 presents the main results, including the asymptotic distributions of the order statistics and rates of convergence. Section~4 contains the proofs of the main results. Appendices~A, B, C provide proofs of auxiliary results and necessary lemmas, theorems used in Section~4.

\section{Model and Additional Notation}

\noindent{\bf Model $M_{[0,1]}$.}
Let $D\subset[0,1]$ be a countable set of possible values of $X_{ij}$, and assume that $|D|>1$.
We consider a round-robin tournament in which, for all $i\neq j$,
\begin{equation}
\label{model:Basic}
X_{ij}+X_{ji}=1,
\quad
X_{ij}\in D.
\end{equation}
In particular, $a\in D$ implies $1-a\in D$.

We assume that the players are equally strong and each value can be attained with positive probability, that is,
\begin{equation}
\label{eq:es}
0<\mathbb{P}\!\left(X_{ij}=a\right)
=
\mathbb{P}\!\left(X_{ji}=a\right),
\quad \text{for all } a\in D.
\end{equation}
Under \eqref{model:Basic}, the assumption \eqref{eq:es} is equivalent to the distribution of $X_{ij}$ being symmetric about $1/2$: for every $a\in D$,
\begin{equation}
\label{eq:symD}
0<\mathbb{P}(X_{ij}=a)=\mathbb{P}(X_{ij}=1-a),
\quad 1\le i<j\le n.
\end{equation}
Consequently, the scores $s_1(n),\ldots,s_n(n)$ are identically distributed.

Under Model $M_{[0,1]}$, let
\begin{equation*}
%\label{eq:mom}
\mu=\mathbb{E}(X_{ij}),\quad \sigma=\sqrt{\mathrm{Var}(X_{ij})}.
\end{equation*}
Let
\begin{equation}\label{eq:ab}
a_n=(2\log n)^{-1/2},\quad
b_n=(2\log n)^{1/2}-\frac{\log\log n+\log(4\pi)}{2(2\log n)^{1/2}},
\end{equation}
and for $t\in\mathbb{R}$ define
\begin{equation}
\label{eq:xnt}
x_n(t)=b_n+a_n t
=(2\log n)^{1/2}+\frac{t-\frac{\log\log n+\log(4\pi)}{2}}{(2\log n)^{1/2}}.
\end{equation}

Under Model $M_{[0,1]}$, the normalized score of player $i$ is
\begin{equation*}
%\label{eq:star-score}
s^{\star}_i(n)=\frac{s_i(n)-(n-1)\mu}{\sqrt{n-1}\,\sigma}.
\end{equation*}
Define the exceedance indicators and their count:
\begin{equation}
\label{eq:W}
I_i^{(n)}(t):=\mathbf{1}\{s^{\star}_i(n)>x_n(t)\},\quad
W_n(t):=\sum_{i=1}^n I_i^{(n)}(t).
\end{equation}
Write
\begin{equation}
\label{eq:pw}
p_n(t):=\mathbb{P}(s^{\star}_i(n)>x_n(t)),\quad
\lambda_n(t):=\mathbb{E}[W_n(t)]=np_n(t).
\end{equation}

The total variation distance between the distributions of random elements $X$ and $Y$ is denoted $d_{\mathrm{TV}}(X,Y)$ and defined by
$$
d_{\mathrm{TV}}(X,Y)=
\frac{1}{2}\sum_{k\ge 0}
\left|\mathbb{P}(Y=k)-\mathbb{P}(X=k)\right|.
$$

Fix an integer $j\ge 0$ and define
\begin{equation*}
%\label{eq:Mnj}
M_{n,j}:=\frac{s_{(n-j)}^{\star}(n)-b_n}{a_n}.
\end{equation*}

The notation $f_n\sim g_n$ denotes asymptotic equivalence, that is,
$$
\lim_{n\to\infty}\frac{f_n}{g_n}=1.
$$

%%%%%%%%%%%%%%%%%%%%%%%%%%%%%%%%%%%%%%%%%%%%%%%%%%%%%%%%
%%%%%%%%%%%%%%%%%%%%%%%%%%%%%%%%%%%%%%%%%%%%%%%%%%%%%%%%
\section{Main Results}

\begin{thm}
\label{thm:TV}
For every fixed $t\in\mathbb{R}$, there exists a constant $C(t)<\infty$ such that
\begin{equation}
\label{eq:TV_Poi_et}
d_{\mathrm{TV}}\bigl(\mathcal L(W_n(t)),\,\mathrm{Poisson}(e^{-t})\bigr)
\le
C(t)\,
\frac{(\log\log n)^2}{\log n},
\quad n\to\infty.
\end{equation}
\end{thm}

That is, $W_n(t)$ converges in total variation to a Poisson random variable with mean $\mathrm{e}^{-t}$, at a rate of order $(\log\log n)^2/\log n$. Convergence in total variation implies convergence in distribution (see, for example, \cite[Exercise~7.2.9, p.~362; Exercise~7.11.16, p.~401]{GS2020}), and therefore we obtain the following corollary.
\medskip

\begin{cor}
\label{cor:Poisson}
For every fixed $t\in\mathbb{R}$ and each fixed integer $k\ge 0$,
$$
\lim_{n\to\infty} \mathbb{P}\!\left(W_n(t)=k\right)
=
\mathrm{e}^{-\mathrm{e}^{-t}}\frac{\mathrm{e}^{-t k}}{k!}.
$$
Moreover, the convergence of $W_n(t)$ to a Poisson random variable with mean $\mathrm{e}^{-t}$ holds in total variation at a rate of order $(\log\log n)^2/\log n$.
\end{cor}

Furthermore, recalling \eqref{eq:W}, we observe that
$$
\{W_n(t)\le j\}=\{s^{\star}_{(n-j)}(n)\le x_n(t)\}.
$$
Combining this observation with Corollary~\ref{cor:Poisson}, we obtain the following corollary.
\medskip

\begin{cor}
\label{cor:cor2}
For every fixed $t\in\mathbb{R}$ and each fixed integer $j\ge 0$,
\begin{equation}
\label{eq:M}
\lim_{n\to\infty}
\mathbb{P}\!\left(
\frac{s_{(n-j)}(n)-(n-1)\mu}{\sqrt{n-1}\,\sigma}
\le x_n(t)
\right)
=\exp(-e^{-t})\sum_{k=0}^{j}\frac{e^{-tk}}{k!},
\end{equation}
where $x_n(t)$ is defined in \eqref{eq:xnt}.
Moreover, the convergence in \eqref{eq:M} holds at a rate of order $(\log\log n)^2/\log n$.
\end{cor}

%%%%%%%%%%%%%%%%%%%%%%%%%%%%%%%%%%%%%%%%%%%%%%%%%%%%%%%%%%%%%%%
\section{Numerics}

We illustrate Theorem~\ref{thm:TV} for the chess round--robin tournament model where $X_{ij} \in \left\{0,\tfrac12,1\right\}$ with
$
\mathbb P(X_{ij}=\tfrac12)=p,\,\,\,\
\mathbb P(X_{ij}=1)=\mathbb P(X_{ji}=0)=\frac{1-p}{2},
$
where $p\in\left\{1/6,\,1/3,\,2/3\right\}$.  
For each unordered pair $\{i,j\}$, the random vector $(X_{ij},X_{ji})$ is generated independently subject to $X_{ij}+X_{ji}=1$.
For each value of
$
n\in\left\{10,\,100,\,500,\,1000,\,10000,\,25000,\,50000,\,100000\right\},$
\,\,
$p\in\left\{1/6,\,1/3,\,2/3\right\},$
and
$t\in\left\{-1,\,0\right\},
$
we simulated the entire tournament, computed the normalized scores
$s_1^{\star}(n),\ldots,s_n^{\star}(n)$, and formed
$
W_n(t)=\sum_{i=1}^n \mathbf 1\{s_i^{\star}(n)>x_n(t)\},
$
where $x_n(t)$ is defined in \eqref{eq:xnt}.  
We repeated the entire procedure $MC$ times, obtained an empirical probability mass function of $W_n(t)$, and used it to estimate the total variation distance
$$
\widehat{d_{\mathrm{TV}}}\!\left(\mathcal L(W_n(t)),\mathrm{Poisson}(\mathrm e^{-t})\right)
=
\frac12\sum_{k\ge0}
\bigl|
\widehat{\mathbb P}(W_n(t)=k)
-
\mathbb P(\mathrm{Poisson}(\mathrm e^{-t})=k)
\bigr|,
$$
where $\widehat{\mathbb P}$ denotes the Monte--Carlo frequencies.
The resulting values are reported in Table~\ref{tab:tv_all_t}.  

It is important to note that each simulation required the generation and storage of an $n\times n$ matrix of match outcomes.  
For $n>10000$, this required the use of high-performance computing resources with a large amount of memory.
For example, for $n=100000$ a single tournament realization required approximately $580$~GB of memory, and $500$ Monte-Carlo replications required about $22$ hours of computation.

Although the accessible values of $n$ are restricted by computational constraints, the results are consistent with the convergence rate asserted in Theorem~\ref{thm:TV}.

\begin{table}[H]
\centering
\caption{Empirical total variation distance between the simulated law of $W_n(t)$
and its Poisson limit for the chess model with $p=1/6,\,1/3,\,2/3$.
Results are shown for two regimes: $t=-1$, and $t=0$.
$MC$ denotes the number of Monte--Carlo replications.}
\label{tab:tv_all_t}
\begin{tabular}{lll|lll|lll|lll|}
\hline
 &  &  & \multicolumn{3}{c}{$t=-1$}
 & \multicolumn{3}{c}{$t=0$}
 \\
\cline{4-6}\cline{7-9}
$n$
& $\dfrac{(\log\log n)^2}{\log n}$
& $MC$
& $p=\tfrac16$ & $p=\tfrac13$ & $p=\tfrac23$
& $p=\tfrac16$ & $p=\tfrac13$ & $p=\tfrac23$ \\
\hline
10     & 0.3021   &10000 & 0.3042 & 0.3821 & 0.2915 & 0.1612  & 0.1223 & 0.1362   \\
100    & 0.5064   &10000 & 0.1706 & 0.1622 &0.1206  & 0.0477  & 0.0749 & 0.0254   \\
500    & 0.5371   &10000 & 0.1360 & 0.1012 &0.0943  & 0.0476  & 0.0615 & 0.0350   \\
1000   & 0.5407   &10000 & 0.1597 & 0.1415 & 0.0932 & 0.0222  & 0.0396 & 0.0159   \\
10000  & 0.5353   &10000 & 0.0955 & 0.1121 & 0.1186 & 0.0310  & 0.0350 & 0.0228  \\
25000  & 0.5293   &1000  & 0.0994 & 0.0979 & 0.1008 & 0.0403  & 0.0238 & 0.0156   \\
50000  & 0.5241   &1000  & 0.1028 & 0.0848 & 0.1278 & 0.0497  & 0.0216 & 0.0322   \\
100000 & 0.5186   &500   & 0.1086 & 0.1166 & 0.1092 & 0.0341  & 0.0236 & 0.0716  \\
\hline
\end{tabular}
\end{table}

%%%%%%%%%%%%%%%%%%%%%%%%%%%%%%%%%%%%%%%%%%%%%%%%%%%%%%%%%%%%%%%%

\section{Proofs}
\begin{proof}[Proof of Theorem~\ref{thm:TV}]
Under Model $M_{[0,1]}$, it follows from \cite[Proposition~1]{MR2022} that the score vector
$\bigl(s_1(n),\ldots,s_n(n)\bigr)$ is negatively associated (see Definition~\ref{def:NA}).
Property~$P_6$ in \cite{JP1983} states that increasing functions defined on disjoint subsets of a negatively associated family are themselves negatively associated.

For each $i=1,\ldots,n$, the indicator
$
I_i^{(n)}(t)=\mathbf{1}\{s_i^\star(n)>x_n(t)\}
$
is an increasing function of the single coordinate $s_i(n)$
(since $s_i^\star(n)$ is obtained from $s_i(n)$ by centering and scaling).
Hence, by Property~$P_6$ in \cite{JP1983},
the family $\bigl(I_1^{(n)}(t),\ldots,I_n^{(n)}(t)\bigr)$ is negatively associated.

Using this fact, and combining Theorem~2.I (p.~30) and Corollary~2.C.2 (p.~36) of \cite{BHJ1992}, we obtain for every fixed $t\in\mathbb{R}$ the following total variation bound:
\begin{align}
\label{eq:TVU}
d_{\mathrm{TV}}\bigl(\mathcal L(W_n(t)),\,\mathrm{Poisson}(\lambda_n(t))\bigr)
&=
\frac{1}{2}\sum_{k\ge 0}
\left|
\mathbb{P}(W_n(t)=k)
-
\mathbb{P}(\mathrm{Poisson}(\lambda_n(t))=k)
\right|
\nonumber\\
&\le
\left(1-\mathrm{e}^{-\lambda_n(t)}\right)
\left(1-\frac{\mathrm{Var}(W_n(t))}{\lambda_n(t)}\right)
\nonumber\\
&=
\left(1-\mathrm{e}^{-\lambda_n(t)}\right)
\left[
\frac{\lambda_n(t)}{n}
-
\frac{n(n-1)\,\mathrm{Cov}\!\left(I_{1}^{(n)}(t), I_{2}^{(n)}(t)\right)}{\lambda_n(t)}
\right]
\nonumber\\
&\le
\frac{\lambda_n(t)}{n}
-
\frac{n(n-1)\,\mathrm{Cov}\!\left(I_{1}^{(n)}(t), I_{2}^{(n)}(t)\right)}{\lambda_n(t)}.
\end{align}
Here the equality follows from the fact that $s_1(n),\ldots,s_n(n)$ are identically distributed and from the definition of $\lambda_n(t)$ in \eqref{eq:pw}.

To evaluate the upper bound, we will need the following two propositions.
\smallskip

\begin{pro}
\label{pro:1}
For every fixed $t\in\mathbb{R}$,
\begin{equation*}
\lambda_n(t)=n\,p_n(t)
=
\mathrm{e}^{-t}
\left(
1+O\!\left(\frac{(\log\log n)^2}{\log n}\right)
\right),
\quad n\to\infty.
\end{equation*}
\end{pro}
\begin{proof}
Appendix \ref{app:B}.
\end{proof}
\smallskip

\begin{pro}
\label{pro:2}
For every fixed $t\in\mathbb{R}$,
\begin{equation*}
\mathrm{Cov}\!\left(I_1^{(n)}(t), I_2^{(n)}(t)\right)
=
-\frac{2\,\mathrm{e}^{-2t}\log n}{n^3}
\left(
1+O\!\left(\frac{(\log\log n)^2}{\log n}\right)
\right),
\quad n\to\infty.
\end{equation*}
\end{pro}
\begin{proof}
Appendix \ref{app:C}.
\end{proof}

By the triangle inequality,
\begin{align}
&
d_{\mathrm{TV}}\bigl(\mathcal L(W_n(t)),\,\mathrm{Poisson}(e^{-t})\bigr)\nonumber\\
&\le
d_{\mathrm{TV}}\bigl(\mathcal L(W_n(t)),\,\mathrm{Poisson}(\lambda_n(t))\bigr)
+
d_{\mathrm{TV}}\bigl(\mathrm{Poisson}(\lambda_n(t)),\,\mathrm{Poisson}(e^{-t})\bigr).
\label{eq:triang_TV}
\end{align}

From Proposition~\ref{pro:1} and Proposition~\ref{pro:2}, we obtain
\begin{align}
-\frac{n(n-1)\,\mathrm{Cov}\!\left(I_{1}^{(n)}(t), I_{2}^{(n)}(t)\right)}{\lambda_n(t)}
&=
\frac{n(n-1)}{\lambda_n(t)}
\cdot
\frac{2e^{-2t}\log n}{n^3}
\left(1+O\!\left(\frac{(\log\log n)^2}{\log n}\right)\right)
\nonumber\\
&=
\frac{2e^{-2t}\log n}{\lambda_n(t)\,n}
\left(1+O\!\left(\frac{(\log\log n)^2}{\log n}\right)\right)
\nonumber\\
&=
\frac{2e^{-t}\log n}{n}
\left(1+O\!\left(\frac{(\log\log n)^2}{\log n}\right)\right).
\label{eq:cov_term_scaled}
\end{align}

Appealing again to Proposition~\ref{pro:1} and using \eqref{eq:cov_term_scaled}, we obtain from \eqref{eq:TVU}
that for every fixed $t\in\mathbb{R}$ there exists a constant $C_1(t)<\infty$ such that
\begin{equation}\label{eq:TV_to_Poi_lambda}
d_{\mathrm{TV}}\bigl(\mathcal L(W_n(t)),\,\mathrm{Poisson}(\lambda_n(t))\bigr)
\leq 
C_1(t)\,\frac{\log n}{n},
\quad n\to\infty.
\end{equation}

\medskip

From Result~\ref{res:TV_Poi_means} in Appendix~\ref{app:A},
for all $\lambda,\mu\ge 0$,
\begin{equation}\label{eq:TVPoi}
d_{\mathrm{TV}}\bigl(\mathrm{Poisson}(\lambda),\,\mathrm{Poisson}(\mu)\bigr)
\le |\lambda-\mu|.
\end{equation}
Hence, by \eqref{eq:TVPoi} and Proposition~\ref{pro:1},
\begin{equation}\label{eq:Poi_mean_mismatch}
d_{\mathrm{TV}}\bigl(\mathrm{Poisson}(\lambda_n(t)),\,\mathrm{Poisson}(e^{-t})\bigr)
\le
|\lambda_n(t)-e^{-t}|
=
e^{-t}\,O\!\left(\frac{(\log\log n)^2}{\log n}\right),
\quad n\to\infty.
\end{equation}

\medskip

Substituting \eqref{eq:TV_to_Poi_lambda} and \eqref{eq:Poi_mean_mismatch} into
\eqref{eq:triang_TV} gives
\begin{align*}
d_{\mathrm{TV}}\bigl(\mathcal L(W_n(t)),\,\mathrm{Poisson}(e^{-t})\bigr)
&\le
C_1(t)\,\frac{\log n}{n}
+
e^{-t}\,O\!\left(\frac{(\log\log n)^2}{\log n}\right).
\end{align*}
Since $\frac{\log n}{n}=o\!\left(\frac{(\log\log n)^2}{\log n}\right)$, we obtain \eqref{eq:TV_Poi_et}.
\end{proof}

\appendix
\section{Auxiliary Results}
\label{app:A}
%%%%%%%%%%%%%%%%%%%%%%%%%%%%%%%%%%%%%%%%%%%%%%%%%%%%%%%%%%
\begin{result}
\label{res:TV_Poi_means}
For all $\lambda,\mu\ge 0$,
\begin{equation*}
d_{\mathrm{TV}}\bigl(\mathrm{Poisson}(\lambda),\,\mathrm{Poisson}(\mu)\bigr)
\le |\lambda-\mu|.
\end{equation*}
\end{result}

\begin{proof}[Proof of Result~\ref{res:TV_Poi_means}]
Assume without loss of generality that $\lambda\le \mu$.
Couple $N\sim\mathrm{Poisson}(\lambda)$ and $M\sim\mathrm{Poisson}(\mu)$ by letting
$M=N+K$, where $K\sim\mathrm{Poisson}(\mu-\lambda)$ is independent of $N$.
Then
$$
d_{\mathrm{TV}}(N,M)
\le \mathbb P(M\neq N)
= \mathbb P(K>0)
= 1-e^{-(\mu-\lambda)}
\le \mu-\lambda
= |\lambda-\mu|,
$$
where the first inequality follows from 
\cite[equation~(1.4), p.~254]{BHJ1992}.
\end{proof}

%%%%%%%%%%%%%%%%%%%%%%%%%%%%%%%%%%%%%%%%%%%%%%%%%%%%%%%%%%

Let $\Phi(t)=\int_{-\infty}^t \varphi(x),dx$, where $\varphi(x)=(2\pi)^{-1/2}e^{-x^2/2}$ denotes the density of the standard normal distribution.
\smallskip

\begin{result}[{\cite[Chapter~VII.7, p.~193]{F1968}}]
\label{eq:mills}
\begin{equation*}
1-\Phi(x)
=
\frac{\varphi(x)}{x}
\left(
1-\frac{1}{x^{2}}+O\!\left(\frac{1}{x^{4}}\right)
\right),
\quad x\to\infty.
\end{equation*}
\end{result}

%%%%%%%%%%%%%%%%%%%%%%%%%%%%%%%%%%%%%%%%%%%%%%%%%%%%%

\begin{lem}
\label{lem:tail_pa}
Let $v_n\to\infty$ and let $u_n$ satisfy
$u_n=v_n+\Delta_n,\,\,\,\Delta_n=o(1).
$

Then, as $n\to\infty$,
\begin{equation}\label{eq:tail_ratio_general}
\frac{1-\Phi(u_n)}{1-\Phi(v_n)}
=
\exp\!\left(-v_n\Delta_n-\frac{\Delta_n^2}{2}\right)
\left(1+O\!\left(\frac{|\Delta_n|}{v_n}+\frac{1}{v_n^2}\right)\right).
\end{equation}
In particular, if $|\Delta_n|=O(v_n/n)$ and $v_n^2=O(\log n)$, then
\begin{equation}\label{eq:tail_ratio_simpler}
\frac{1-\Phi(u_n)}{1-\Phi(v_n)}
=
1+O\!\left(\frac{v_n^2}{n}+\frac{1}{v_n^2}\right),
\end{equation}
and hence
\begin{equation}
\label{eq:final}
1-\Phi(u_n)
=
(1-\Phi(v_n))
\left(1+O\!\left(\frac{1}{\log n}\right)\right).
\end{equation}
\end{lem}
\smallskip

\begin{proof}[Proof of Lemma \ref{lem:tail_pa}]
We use Result~\ref{eq:mills}.
For $x\to\infty$ write
\[
1-\Phi(x)
=
\frac{\varphi(x)}{x}\,m(x),
\qquad
m(x):=1-\frac{1}{x^{2}}+O\!\left(\frac{1}{x^{4}}\right).
\]
Since  $m(x)=1+O(x^{-2})$ and $u_n\sim v_n\to\infty$, we have 
\begin{equation}\label{eq:m_ratio}
\frac{m(u_n)}{m(v_n)}=1+O\!\left(\frac{1}{v_n^2}\right),
\quad n\to\infty.
\end{equation}

Now take the ratio:
\begin{equation}
\label{eq:ratio}
\frac{1-\Phi(u_n)}{1-\Phi(v_n)}
=
\frac{\varphi(u_n)}{\varphi(v_n)}\cdot \frac{v_n}{u_n}\cdot \frac{m(u_n)}{m(v_n)}.
\end{equation}

First,
\begin{align}
\label{eq:first}
&
\frac{\varphi(u_n)}{\varphi(v_n)}
=
\exp\!\left(-\frac{u_n^2-v_n^2}{2}\right)
=
\exp\!\left(-\frac{(v_n+\Delta_n)^2-v_n^2}{2}\right)
=
\exp\!\left(-v_n\Delta_n-\frac{\Delta_n^2}{2}\right).
\end{align}
Second,
\begin{align}
\label{eq:second}
\frac{v_n}{u_n}=\frac{v_n}{v_n+\Delta_n}
=
\left(1+\frac{\Delta_n}{v_n}\right)^{-1}
=
1+O\!\left(\frac{|\Delta_n|}{v_n}\right),
\quad n\to\infty.
\end{align}
Indeed, $u_n/v_n=1+\Delta_n/v_n\to 1$ since $\Delta_n=o(1)$ and $v_n\to\infty$,
so $u_n\sim v_n$.

Substituting these three evaluations \eqref{eq:m_ratio}, \eqref{eq:first}, \eqref{eq:second} into \eqref{eq:ratio} yields
\eqref{eq:tail_ratio_general}.

Now, assume $|\Delta_n|=O(v_n/n)$ and $v_n^2=O(\log n)$.
Then $|\Delta_n|/v_n=O(1/n)$ and $1/v_n^2=O(1/\log n)$, while
$$
v_n|\Delta_n|=O\!\left(\frac{v_n^2}{n}\right)=O\!\left(\frac{\log n}{n}\right),
\,\,\,
\Delta_n^2=O\!\left(\frac{v_n^2}{n^2}\right)=o\!\left(\frac{v_n^2}{n}\right).
$$
Therefore the exponential term satisfies
$$
\exp\!\left(-v_n\Delta_n-\frac{\Delta_n^2}{2}\right)
=
1+O(v_n|\Delta_n|)+O(\Delta_n^2)
=
1+O\!\left(\frac{v_n^2}{n}\right),
$$
and multiplying by $1+O(|\Delta_n|/v_n+1/v_n^2)$ we obtain
$$
\frac{1-\Phi(u_n)}{1-\Phi(v_n)}
=
1+O\!\left(\frac{v_n^2}{n}+\frac{1}{v_n^2}\right),
$$
which proves \eqref{eq:tail_ratio_simpler}. Finally, under $v_n^2=O(\log n)$ we have
$1/v_n^2=O(1/\log n)$ and $v_n^2/n=O(\log n/n)=o(1/\log n)$, hence
$$
\frac{1-\Phi(u_n)}{1-\Phi(v_n)}
=
1+O\!\left(\frac{1}{\log n}\right),
$$
which yields \eqref{eq:final}.

\end{proof}

%%%%%%%%%%%%%%%%%%%%%%%%%%%%%%%%%%%%%%%%%%%%%%%%%%%%
 
We will also need the following Cram\'er-type moderate deviation theorem.
\begin{thm}[{\cite[Theorem~1, \S VIII.2, p.~269; see also p.~323]{P1975}}]
\label{th:petrov}
Let $Y_1,Y_2,\ldots$ be i.i.d.\ random variables such that
$\mathbb{E}Y_1=0$ and $\mathbb{E}Y_1^2=1$.
Assume Cram\'er's condition: there exists $H>0$ such that
$\mathbb{E}e^{tY_1}<\infty$ for all $|t|<H$.
Let $S_n:=\sum_{k=1}^n Y_k$ and let $x=x_n\ge 0$ satisfy $x=o(\sqrt{n})$.
Then
\begin{equation*}
%\label{eq:petrov}
\frac{\mathbb{P}\!\left(S_n/\sqrt{n}>x\right)}{1-\Phi(x)}
=
\exp\!\left(\frac{x^3}{\sqrt{n}}\,
\eta\!\left(\frac{x}{\sqrt{n}}\right)\right)
\left(1+O\!\left(\frac{x+1}{\sqrt{n}}\right)\right),
\end{equation*}
where $\eta(u)=\sum_{m\ge 0} a_m u^m$ is a power series whose coefficients
depend on the cumulants of $Y_1$ and which converges for all sufficiently small
$|u|$.
\end{thm}

%%%%%%%%%%%%%%%%%%%%%%%%%%%%%%%%%%%%%%%%%%%%%%%%%%%%
\begin{lem}
\label{lem:lemma2_uq}
Fix $t\in\mathbb R$ and write $x_n:=x_n(t)$.
Let $q_n(u)=\mathbb P(R_1>u)$, where
$$
R_1=\frac{1}{\sqrt{n-2}}\sum_{j=3}^n Y_{1j},
\qquad
Y_{ij}=\frac{X_{ij}-\mu}{\sigma}.
$$
Then the following hold.

\smallskip
\noindent (a) 
There exists a constant $C(t)<\infty$ such that, uniformly over all real sequences
$\{u_n\}$ satisfying
\begin{equation}
\label{eq:un_window}
u_n=x_n+\Delta_n,\qquad |\Delta_n|\le \frac{C(t)}{\sqrt n},
\end{equation}
we have
\begin{equation}
\label{eq:qn_vs_normal}
q_n(\alpha_n u_n)
=
(1-\Phi(u_n))
\left(
1+O\!\left(\frac{(\log\log n)^2}{\log n}\right)
\right),
\quad n\to\infty.
\end{equation}

\smallskip
\noindent (b) 
Under \eqref{eq:un_window},
\begin{equation}
\label{eq:qn_diff}
q_n(\alpha_n(x_n-\Delta_n))
-
q_n(\alpha_n(x_n+\Delta_n))
=
2\Delta_n\,\varphi(x_n)
\left(
1+O\!\left(\frac{(\log\log n)^2}{\log n}\right)
\right),
\quad n\to\infty,
\end{equation}
uniformly in $|\Delta_n|\le C(t)/\sqrt n$.
\end{lem}

\begin{proof}[Proof of Lemma~\ref{lem:lemma2_uq}]

\medskip
By definition,
$
R_1=\frac{1}{\sqrt{n-2}}\sum_{j=3}^n Y_{1j},
$
where the summands are i.i.d., mean $0$, variance $1$, and bounded.
Hence Cramér’s condition holds, and Petrov’s theorem
applies. Let $m=n-2$.  
Petrov’s Theorem \ref{th:petrov} yields, uniformly for $u=o(\sqrt m)$,
\begin{equation*}
%\label{eq:petrov_R1}
\frac{\mathbb P(R_1>u)}{1-\Phi(u)}
=
\exp\!\left(\frac{u^3}{\sqrt m}\eta\!\left(\frac{u}{\sqrt m}\right)\right)
\left(1+O\!\left(\frac{u+1}{\sqrt m}\right)\right),
\quad m\to\infty,
\end{equation*}
where $\eta(\cdot)$ is analytic near $0$.

Now in our regime $u=\alpha_n u_n$ with $u_n=x_n+\Delta_n$ and
$|\Delta_n|\le C/\sqrt n$, we have $u_n\sim x_n\sim \sqrt{2\log n}$ and
$\alpha_n=1+O(1/n)$, hence
\begin{equation}
\label{eq:petrov_smallness}
u=\alpha_n u_n = O(\sqrt{\log n}),
\qquad
\frac{u}{\sqrt m}=O\!\left(\frac{\sqrt{\log n}}{\sqrt n}\right)\to 0.
\end{equation}
Thus $u=o(\sqrt m)$ and Petrov Theorem \ref{th:petrov} applies.

Moreover, from \eqref{eq:petrov_smallness},
\[
\frac{u^3}{\sqrt m}
=
O\!\left(\frac{(\log n)^{3/2}}{\sqrt n}\right),
\qquad
\frac{u+1}{\sqrt m}
=
O\!\left(\frac{\sqrt{\log n}}{\sqrt n}\right).
\]
Both are $o((\log\log n)^2/\log n)$, because
\[
\frac{(\log n)^{3/2}}{\sqrt n}\Big/\frac{(\log\log n)^2}{\log n}
=
\frac{(\log n)^{5/2}}{\sqrt n(\log\log n)^2}\to 0.
\]
Hence, 
\begin{equation}
\label{eq:tail}
\mathbb P(R_1>u)
=
(1-\Phi(u))
\left(
1+O\!\left(\frac{(\log\log n)^2}{\log n}\right)
\right).
\end{equation}

\medskip

We now apply Lemma~\ref{lem:tail_pa} with
$$
v_n=u_n,\qquad u_n^{\star}:=\alpha_n u_n,
\qquad
\Delta_n^{\star}=u_n^{\star}-v_n=(\alpha_n-1)u_n.
$$
We have $\alpha_n-1=O(1/n)$ and $u_n=O(\sqrt{\log n})$, so
\begin{equation*}
|\Delta_n^{\star}|=O\!\left(\frac{\sqrt{\log n}}{n}\right)=o(1).
\end{equation*}
Also $v_n^2=u_n^2=O(\log n)$ and
\[
\frac{|\Delta_n^{\star}|}{v_n}
=
O\!\left(\frac{1}{n}\right),
\qquad
\frac{1}{v_n^2}
=
O\!\left(\frac{1}{\log n}\right).
\]
Therefore Lemma~\ref{lem:tail_pa} gives
\begin{equation}
\label{eq:tail_alpha_transfer}
\frac{1-\Phi(\alpha_n u_n)}{1-\Phi(u_n)}
=
1+O\!\left(\frac{1}{\log n}\right),
\quad n\to\infty,
\end{equation}
uniformly in $|\Delta_n|\le C/\sqrt n$.

Combining \eqref{eq:tail} at $u=\alpha_n u_n$ with
\eqref{eq:tail_alpha_transfer} yields
\begin{align*}
&
q_n(\alpha_n u_n)
=
(1-\Phi(\alpha_n u_n))
\left(1+O\!\left(\frac{(\log\log n)^2}{\log n}\right)\right)\\
&
=
(1-\Phi(u_n))
\left(1+O\!\left(\frac{(\log\log n)^2}{\log n}\right)\right),
\end{align*}
which is \eqref{eq:qn_vs_normal}. This proves part (a).

\medskip

Apply part (a) with $u_n=x_n\pm\Delta_n$ to get
\begin{align}
q_n(\alpha_n(x_n-\Delta_n))
&=
(1-\Phi(x_n-\Delta_n))
\left(1+O\!\left(\frac{(\log\log n)^2}{\log n}\right)\right),
\label{eq:qminus}\\
q_n(\alpha_n(x_n+\Delta_n))
&=
(1-\Phi(x_n+\Delta_n))
\left(1+O\!\left(\frac{(\log\log n)^2}{\log n}\right)\right),
\label{eq:qplus}
\end{align}
uniformly for $|\Delta_n|\le C/\sqrt n$.

Subtract \eqref{eq:qplus} from \eqref{eq:qminus}:
\begin{align}
&
q_n(\alpha_n(x_n-\Delta_n))-q_n(\alpha_n(x_n+\Delta_n))\nonumber\\
&=
\Bigl[(1-\Phi(x_n-\Delta_n))-(1-\Phi(x_n+\Delta_n))\Bigr]
\left(1+O\!\left(\frac{(\log\log n)^2}{\log n}\right)\right).
\label{eq:diff_reduce}
\end{align}

Now,
\begin{align}
&
(1-\Phi(x_n-\Delta_n))-(1-\Phi(x_n+\Delta_n))\nonumber\\
&
=
\Phi(x_n+\Delta_n)-\Phi(x_n-\Delta_n)
=
\int_{x_n-\Delta_n}^{x_n+\Delta_n}\varphi(u)\,du.
\label{eq:int_phi}
\end{align}
We estimate the integral by comparing $\varphi(u)$ to $\varphi(x_n)$.

$$
\frac{\varphi(u)}{\varphi(x_n)}
=
\exp\!\left(-\frac{u^2-x_n^2}{2}\right)
=
\exp\!\left(-(u-x_n)x_n-\frac{(u-x_n)^2}{2}\right).
$$
Since $|u-x_n|\le |\Delta_n|\le C/\sqrt n$ and $x_n\sim\sqrt{2\log n}$,
\[
|(u-x_n)x_n| \le \frac{C\sqrt{\log n}}{\sqrt n}\to 0,
\qquad
(u-x_n)^2\le \frac{C^2}{n}\to 0.
\]
Hence uniformly in the integration range,
\begin{equation}
\label{eq:phi_flat}
\varphi(u)=\varphi(x_n)\left(1+O\!\left(\frac{\sqrt{\log n}}{\sqrt n}\right)\right).
\end{equation}

Substitute \eqref{eq:phi_flat} into \eqref{eq:int_phi}:
\[
\int_{x_n-\Delta_n}^{x_n+\Delta_n}\varphi(u)\,du
=
2\Delta_n\,\varphi(x_n)\left(1+O\!\left(\frac{\sqrt{\log n}}{\sqrt n}\right)\right).
\]
Since $\sqrt{\log n}/\sqrt n=o((\log\log n)^2/\log n)$, and with \eqref{eq:diff_reduce} this yields
\eqref{eq:qn_diff}.
\end{proof}

%%%%%%%%%%%%%%%%%%%%%%%%%%%%%%%%%%%%%%%%%%%%%%%%%%%%%%%%%%%%%%%%%%%%
% ------------------------------------------------------------------
%
% ------------------------------------------------------------------

\begin{lem}
\label{lem:lemma3_var_bound}
Fix $t\in\mathbb R$ and let $x_n:=x_n(t)$.
Let $q_n(u)=\mathbb P(R_1>u)$, where
$$
R_1=\frac{1}{\sqrt{n-2}}\sum_{j=3}^n Y_{1j},
\qquad
Y_{ij}=\frac{X_{ij}-\mu}{\sigma}.
$$
Let
$$
\alpha_n:=\sqrt{\frac{n-1}{n-2}},
\,\,\,\,
A_n(y):=q_n\!\left(\alpha_n\left(x_n-\frac{y}{\sqrt{n-1}}\right)\right),
\,\,\,\,
B_n(y):=q_n\!\left(\alpha_n\left(x_n+\frac{y}{\sqrt{n-1}}\right)\right),
$$
and
$$
M_n(y):=\frac{A_n(y)+B_n(y)}{2}.
$$
Let $Y:=Y_{12}=(X_{12}-\mu)/\sigma$. Then there exists $C(t)<\infty$ such that
\begin{equation}
\label{eq:VarMn_bound_clean}
\mathrm{Var}(M_n(Y))
\le
C(t)\left(\frac{x_n\varphi(x_n)}{n}\right)^2,
\quad n\ge 3.
\end{equation}
In particular,
\begin{equation}
\label{eq:VarMn_order_clean}
\mathrm{Var}(M_n(Y))
=
O\!\left(\frac{(\log n)^2}{n^4}\right),
\quad n\to\infty.
\end{equation}
\end{lem}

\begin{proof}[Proof of Lemma~\ref{lem:lemma3_var_bound}]
Fix $t\in\mathbb R$ and let $x_n=x_n(t)$.
All implicit constants below may depend on $t$ but not on $n$.

\medskip

Since $X_{12}\in[0,1]$ and $\sigma>0$ is fixed, $Y=(X_{12}-\mu)/\sigma$ is bounded:
there exists $K<\infty$ such that
$
|Y|\le K \quad\text{a.s.}
$
Define 
\begin{equation*}
\Delta:=\frac{Y}{\sqrt{n-1}},
\qquad\text{so that}\qquad
|\Delta|\le \frac{K}{\sqrt{n-1}}\le \frac{K}{\sqrt n}.
\end{equation*}
Thus $|\Delta|\le K/\sqrt n$, so Lemma~\ref{lem:lemma2_uq}(a) applies
provided its window constant $C(t)$ is chosen $\ge K$ (which we may assume by enlarging $C(t)$ if needed).

\medskip

Apply Lemma~\ref{lem:lemma2_uq}(a) with $u_n=x_n\pm \Delta$.
There exists a \emph{deterministic} sequence $\varepsilon_n=\varepsilon_n(t)$ such that
\begin{equation}
\label{eq:eps_def}
\varepsilon_n=O\!\left(\frac{(\log\log n)^2}{\log n}\right),
\quad n\to\infty,
\end{equation}
and, uniformly for all realizations of $Y$ (equivalently, uniformly in $|\Delta|\le K/\sqrt n$),
\begin{align*}
A_n(Y)
&=
\bigl(1-\Phi(x_n-\Delta)\bigr)\,(1+\varepsilon_n),
\\
B_n(Y)
&=
\bigl(1-\Phi(x_n+\Delta)\bigr)\,(1+\varepsilon_n).
\end{align*}
Consequently,
\begin{equation*}
M_n(Y)
=
(1+\varepsilon_n)\,\widetilde M_n(\Delta),
\quad
\widetilde M_n(\Delta)
:=
\frac{(1-\Phi(x_n-\Delta))+(1-\Phi(x_n+\Delta))}{2}.
\end{equation*}

\medskip

Therefore,
\begin{equation}
\label{eq:Var_scale_eps}
\mathrm{Var}(M_n(Y))
=
(1+\varepsilon_n)^2\,\mathrm{Var}\bigl(\widetilde M_n(\Delta)\bigr).
\end{equation}
Therefore it suffices to bound $\mathrm{Var}(\widetilde M_n(\Delta))$.

\medskip

Define $g(u):=1-\Phi(u)$. Then $g'(u)=-\varphi(u)$ and $g''(u)=-\varphi'(u)$.
Taylor's theorem gives: for each $\Delta$,
there exist $\theta_1,\theta_2\in(-1,1)$ such that
\begin{align}
&
g(x_n-\Delta)=
g(x_n)-\Delta g'(x_n)+\frac{\Delta^2}{2}g''(x_n-\theta_1\Delta)\nonumber\\
&
=
g(x_n)+\Delta\varphi(x_n)-\frac{\Delta^2}{2}\varphi'(x_n-\theta_1\Delta),
\label{eq:g_minus_Taylor}\\
&
g(x_n+\Delta)
=
g(x_n)+\Delta g'(x_n)+\frac{\Delta^2}{2}g''(x_n+\theta_2\Delta)\nonumber\\
&
=
g(x_n)-\Delta\varphi(x_n)-\frac{\Delta^2}{2}\varphi'(x_n+\theta_2\Delta).
\label{eq:g_plus_Taylor}
\end{align}
Averaging \eqref{eq:g_minus_Taylor} and \eqref{eq:g_plus_Taylor} yields:
\begin{equation}
\label{eq:even_Taylor_exact}
\widetilde M_n(\Delta)
=
g(x_n)
-\frac{\Delta^2}{4}\Bigl(\varphi'(x_n-\theta_1\Delta)+\varphi'(x_n+\theta_2\Delta)\Bigr).
\end{equation}

Therefore,
\begin{equation*}
\mathrm{Var}\bigl(\widetilde M_n(\Delta)\bigr)
=
\mathrm{Var}\!\left(
\frac{\Delta^2}{4}\Bigl(\varphi'(x_n-\theta_1\Delta)+\varphi'(x_n+\theta_2\Delta)\Bigr)
\right).
\end{equation*}

\medskip

Recall $\varphi'(u)=-u\varphi(u)$.
We claim that for all sufficiently large $n$ (depending on $t$), there exists $C<\infty$ such that
\begin{equation}
\label{eq:phiprime_uniform_bound}
\sup_{|\theta|\le 1}\ \sup_{|\Delta|\le K/\sqrt n}\ 
\bigl|\varphi'(x_n+\theta\Delta)\bigr|
\le
C\,x_n\varphi(x_n).
\end{equation}
Indeed, for such $\theta,\Delta$,
$$
\bigl|\varphi'(x_n+\theta\Delta)\bigr|
=
|x_n+\theta\Delta|\,\varphi(x_n+\theta\Delta)
\le
(|x_n|+|\Delta|)\,\varphi(x_n+\theta\Delta).
$$
Moreover,
$$
\frac{\varphi(x_n+\theta\Delta)}{\varphi(x_n)}
=
\exp\!\left(-\theta x_n\Delta-\frac{\theta^2\Delta^2}{2}\right).
$$
Since $|x_n|=O(\sqrt{\log n})$ and $|\Delta|\le K/\sqrt n$, we have
$|x_n\Delta|\le K\sqrt{\log n}/\sqrt n\to 0$ and $\Delta^2\le K^2/n\to 0$,
so this ratio is bounded by positive constant for all large $n$.
Also $|\Delta|=o(x_n)$, hence $|x_n|+|\Delta|\le 2x_n$ for all large $n$.
Combining these bounds yields \eqref{eq:phiprime_uniform_bound}.

\medskip

Let
$$
H_n(\Delta):=\frac{1}{4}\Bigl(\varphi'(x_n-\theta_1\Delta)+\varphi'(x_n+\theta_2\Delta)\Bigr).
$$
By \eqref{eq:phiprime_uniform_bound}, for all large $n$ and all admissible $\Delta$,
\begin{equation}
\label{eq:Hn_bound}
|H_n(\Delta)|
\le
C\,x_n\varphi(x_n).
\end{equation}
Then \eqref{eq:even_Taylor_exact} rewrites as
$$
\widetilde M_n(\Delta)
=
g(x_n)-\Delta^2\,H_n(\Delta),
$$
hence, 
$$
\mathrm{Var}\bigl(\widetilde M_n(\Delta)\bigr)
=
\mathrm{Var}\bigl(\Delta^2 H_n(\Delta)\bigr)
\le
\mathbb E\bigl[(\Delta^2 H_n(\Delta))^2\bigr].
$$
Using \eqref{eq:Hn_bound},
\begin{equation}
\label{eq:Var_tilde_bound}
\mathrm{Var}\bigl(\widetilde M_n(\Delta)\bigr)
\le
C\,x_n^2\varphi(x_n)^2\,\mathbb E[\Delta^4].
\end{equation}
Now $\Delta=Y/\sqrt{n-1}$, so
\begin{equation}
\label{eq:Delta4}
\mathbb E[\Delta^4]
=
\frac{\mathbb E[Y^4]}{(n-1)^2}
\le
\frac{C}{n^2},
\end{equation}
since $Y$ is bounded.
Substituting \eqref{eq:Delta4} into \eqref{eq:Var_tilde_bound} yields
\begin{equation}
\label{eq:Var_tilde_final}
\mathrm{Var}\bigl(\widetilde M_n(\Delta)\bigr)
\le
C\left(\frac{x_n\varphi(x_n)}{n}\right)^2.
\end{equation}

\medskip

Insert \eqref{eq:Var_tilde_final} into \eqref{eq:Var_scale_eps}:
$$
\mathrm{Var}(M_n(Y))
\le
(1+\varepsilon_n)^2\,
C\left(\frac{x_n\varphi(x_n)}{n}\right)^2.
$$
Since $\varepsilon_n\to 0$ by \eqref{eq:eps_def}, $(1+\varepsilon_n)^2$ is bounded,
and we obtain \eqref{eq:VarMn_bound_clean}.

From \eqref{eq:phi_xnt_asymp}, we have
$$
\varphi(x_n)
=\frac{e^{-t}\sqrt{2\log n}}{n}
\left(1+O\!\left(\frac{(\log\log n)^2}{\log n}\right)\right),
$$
and since $x_n=O(\sqrt{\log n})$ we have $x_n\varphi(x_n)=O(\log n/n)$.
Substituting into \eqref{eq:VarMn_bound_clean} gives \eqref{eq:VarMn_order_clean}.
\end{proof}

%%%%%%%%%%%%%%%%%%%%%%%%%%%%%%%%%%%%%%%%%%%%%%%%%%%%%%%%%%%%

\section{Proof of Proposition~\ref{pro:1}}
\label{app:B}
\begin{proof}[Proof of Proposition~\ref{pro:1}]
Fix $t\in\mathbb{R}$.\\ Recall from \eqref{eq:pw} that
$
p_n(t)=\mathbb{P}\!\left(s_1^\star(n)>x_n(t)\right),\,\,\,
\lambda_n(t)=n\,p_n(t)
$
and 
$
s_1(n)=\sum_{j=2}^{n} X_{1j}.
$
Define
$$
Y_j:=\frac{X_{1,j+1}-\mu}{\sigma},\qquad j=1,2,\ldots .
$$
Then $\{Y_j\}$ are i.i.d., $\mathbb{E}Y_1=0$, $Var(Y_1)=1$, and
$$
s_1^\star(n)=\frac{s_1(n)-(n-1)\mu}{\sqrt{n-1}\,\sigma}
=\frac{1}{\sqrt{n-1}}\sum_{j=1}^{n-1} Y_j.
$$

Therefore,
\begin{equation}\label{eq:pn_as_sum}
p_n(t)=
\mathbb{P}\!\left(\frac{1}{\sqrt{n-1}}\sum_{j=1}^{n-1}Y_j>x_n(t)\right).
\end{equation}

Moreover, since $X_{ij}\in D\subset[0,1]$ is bounded, the Cram\'er's condition in Theorem~\ref{th:petrov} holds.

From \eqref{eq:xnt} we have $x_n(t)\sim \sqrt{2\log n}$ as $n\to\infty$, hence
$
x_n(t)=o(\sqrt{n-1})\qquad (n\to\infty).
$
Thus Theorem~\ref{th:petrov} is applicable. Applying Theorem~\ref{th:petrov} to \eqref{eq:pn_as_sum} yields
\begin{equation*}
\frac{p_n(t)}{1-\Phi(x_n(t))}
=
\exp\!\left(\frac{x_n(t)^3}{\sqrt{n-1}}\,
\eta\!\left(\frac{x_n(t)}{\sqrt{n-1}}\right)\right)
\left(1+O\!\left(\frac{x_n(t)+1}{\sqrt{n-1}}\right)\right),
\quad n\to\infty,
\end{equation*}
where $\eta(\cdot)$ is analytic near $0$ (hence bounded in a neighborhood of $0$).
Since $x_n(t)=O(\sqrt{\log n})$, we have
\[
\frac{x_n(t)^3}{\sqrt{n-1}}=O\!\left(\frac{(\log n)^{3/2}}{\sqrt{n}}\right)\to 0,
\qquad
\frac{x_n(t)+1}{\sqrt{n-1}}=O\!\left(\frac{\sqrt{\log n}}{\sqrt{n}}\right)\to 0,
\]
and hence
\begin{equation}\label{eq:petrov_factor_1}
\frac{p_n(t)}{1-\Phi(x_n(t))}
=
1+O\!\left(\frac{(\log n)^{3/2}}{\sqrt{n}}\right),
\quad n\to\infty.
\end{equation}
Note that $(\log n)^{3/2}/\sqrt n = o((\log\log n)^2/\log n)$.

\medskip

By Result~\ref{eq:mills},
\begin{equation}\label{eq:mills_use}
1-\Phi(x)=\frac{\varphi(x)}{x}\left(1+O\!\left(\frac{1}{x^2}\right)\right),
\quad x\to\infty.
\end{equation}
Since $x_n(t)\sim \sqrt{2\log n}\to\infty$, \eqref{eq:mills_use} applies with $x=x_n(t)$:
\begin{equation}\label{eq:tail_in_terms_phi}
1-\Phi(x_n(t))
=\frac{\varphi(x_n(t))}{x_n(t)}\left(1+O\!\left(\frac{1}{x_n(t)^2}\right)\right)
=\frac{\varphi(x_n(t))}{x_n(t)}\left(1+O\!\left(\frac{1}{\log n}\right)\right).
\end{equation}
\bigskip

Recall \eqref{eq:ab} and define  $u_n:=\sqrt{2\log n}$ and $c_n:=\log\log n+\log(4\pi)$, so that
$
b_n=u_n-\frac{c_n}{2u_n},\,\,\, a_n=\frac{1}{u_n}.$
A direct expansion gives
\begin{align*}
x_n(t)^2
&=(b_n+a_nt)^2
=b_n^2+2a_n b_n\,t+a_n^2 t^2\\
&=\left(u_n-\frac{c_n}{2u_n}\right)^2
+2\left(\frac{1}{u_n}\right)\left(u_n-\frac{c_n}{2u_n}\right)t
+O\!\left(\frac{1}{\log n}\right)\\
&=u_n^2-c_n+2t+O\!\left(\frac{(\log\log n)^2}{\log n}\right)
\quad (n\to\infty),
\end{align*}
since $u_n^2=2\log n$ and $c_n^2/u_n^2=O((\log\log n)^2/\log n)$.
Therefore,
\begin{align*}
&
\varphi(x_n(t))
=(2\pi)^{-1/2}\exp\!\left(-\frac{x_n(t)^2}{2}\right)\\
&
=(2\pi)^{-1/2}\exp\!\left(-\log n+\frac{c_n}{2}-t+O\!\left(\frac{(\log\log n)^2}{\log n}\right)\right).
\end{align*}
Using $\exp(c_n/2)=\exp((\log\log n)/2)\exp((\log(4\pi))/2)=\sqrt{\log n}\,\sqrt{4\pi}$,
we obtain
\begin{equation}\label{eq:phi_xnt_asymp}
\varphi(x_n(t))
=\frac{e^{-t}\sqrt{2\log n}}{n}\left(1+O\!\left(\frac{(\log\log n)^2}{\log n}\right)\right).
\end{equation}

Substitute \eqref{eq:phi_xnt_asymp} into \eqref{eq:tail_in_terms_phi} and use
$x_n(t)\sim \sqrt{2\log n}$, we obtain
\begin{equation}\label{eq:tail_asymp_final}
1-\Phi(x_n(t))
=\frac{e^{-t}}{n}\left(1+O\!\left(\frac{(\log\log n)^2}{\log n}\right)\right).
\end{equation}

From \eqref{eq:petrov_factor_1} and \eqref{eq:tail_asymp_final},
$$
p_n(t)
=
\left(1-\Phi(x_n(t))\right)
\left(1+O\!\left(\frac{(\log n)^{3/2}}{\sqrt{n}}\right)\right)
=
\frac{e^{-t}}{n}
\left(1+O\!\left(\frac{(\log\log n)^2}{\log n}\right)\right),
$$
and therefore
\[
\lambda_n(t)=n\,p_n(t)=e^{-t}\left(1+O\!\left(\frac{(\log\log n)^2}{\log n}\right)\right),
\quad n\to\infty.
\]
\end{proof}

\section{Proof of Proposition~\ref{pro:2}}
\label{app:C}
\begin{proof}[Proof of Proposition~\ref{pro:2}]

Fix $t\in\mathbb R$ and let
$
x_n:=x_n(t).
$
Recall that
$$
I_i^{(n)}(t)=\mathbf 1\{s_i^\star(n)>x_n\},
\quad
p_n(t)=\mathbb P(I_1^{(n)}(t)=1),
\quad
\lambda_n(t)=np_n(t).
$$
We shall compute
$$
\mathrm{Cov}\!\left(I_1^{(n)}(t),I_2^{(n)}(t)\right)
=
\mathbb P(I_1^{(n)}(t)=1,I_2^{(n)}(t)=1)-p_n(t)^2.
$$

%%%%%%%%%%%%%%%%%%%%%%%%%%%%%%%%%%%%%%%%%%%%%%%%%%%%%%%%%%%%%%%%%%%%

Define 
$$
Y_{ij}:=\frac{X_{ij}-\mu}{\sigma}, \qquad i\ne j.
$$
Under Model $M_{[0,1]}$, we have $\mu=\mathbb E(X_{ij})=1/2$ by 
\eqref{eq:symD}. Also, $X_{ij}\in[0,1]$ implies $|Y_{ij}|\le 1/\sigma$ is bounded.

For player $1$,
$$
s_1^\star(n)
=\frac{1}{\sqrt{n-1}}\sum_{j=2}^n Y_{1j}
=
\frac{1}{\sqrt{n-1}}\Bigl(Y_{12}+\sum_{j=3}^n Y_{1j}\Bigr).
$$
For player $2$,
$$
s_2^\star(n)
=\frac{1}{\sqrt{n-1}}\sum_{j\ne 2} Y_{2j}
=
\frac{1}{\sqrt{n-1}}\Bigl(Y_{21}+\sum_{j=3}^n Y_{2j}\Bigr).
$$
Using $X_{21}=1-X_{12}$ and $\mu=1/2$, we get
$$
Y_{21}=\frac{X_{21}-1/2}{\sigma}=\frac{(1-X_{12})-1/2}{\sigma}
=
-\frac{X_{12}-1/2}{\sigma}
=
-Y_{12}.
$$
Setting,
$
Y:=Y_{12},
$
we obtain 
\begin{equation}
\label{eq:sstar_representation}
s_1^\star(n)=\frac{1}{\sqrt{n-1}}\Bigl(Y+\sum_{j=3}^n Y_{1j}\Bigr),
\quad
s_2^\star(n)=\frac{1}{\sqrt{n-1}}\Bigl(-Y+\sum_{j=3}^n Y_{2j}\Bigr).
\end{equation}

Now define:
$$
R_1:=\frac{1}{\sqrt{n-2}}\sum_{j=3}^n Y_{1j},
\quad
R_2:=\frac{1}{\sqrt{n-2}}\sum_{j=3}^n Y_{2j}.
$$
Then \eqref{eq:sstar_representation} becomes
\begin{equation}
\label{eq:sstar_R}
s_1^\star(n)=\sqrt{\frac{n-2}{n-1}}\,R_1+\frac{Y}{\sqrt{n-1}},
\quad
s_2^\star(n)=\sqrt{\frac{n-2}{n-1}}\,R_2-\frac{Y}{\sqrt{n-1}}.
\end{equation}

\medskip

The collection of match vectors $\{(X_{ij},X_{ji}):1\le i<j\le n\}$ is independent by assumption.
Therefore:
\begin{itemize}
\item $Y=Y_{12}$ depends only on the match $(1,2)$;
\item $\{Y_{1j}:j=3,\dots,n\}$ depends only on matches $(1,j)$, $j\ge3$;
\item $\{Y_{2j}:j=3,\dots,n\}$ depends only on matches $(2,j)$, $j\ge3$;
\end{itemize}
and these three families correspond to disjoint sets of matches, hence are independent.
Consequently,
\begin{equation}
\label{eq:indep}
Y,\ R_1,\ R_2 \text{ are independent, and } R_1\stackrel{d}=R_2.
\end{equation}

Introduce the notations
$$
\alpha_n:=\sqrt{\frac{n-1}{n-2}},
\qquad
h_n(y):=\frac{y}{\sqrt{n-1}}.
$$
Then from \eqref{eq:sstar_R},
\begin{align}
\{s_1^\star(n)>x_n\}
&=
\left\{R_1>\alpha_n\bigl(x_n-h_n(Y)\bigr)\right\},
\label{eq:event1}\\
\{s_2^\star(n)>x_n\}
&=
\left\{R_2>\alpha_n\bigl(x_n+h_n(Y)\bigr)\right\}.
\label{eq:event2}
\end{align}

Define 
\begin{equation*}
q_n(u):=\mathbb P(R_1>u),\qquad u\in\mathbb R.
\end{equation*}

Conditioning on $Y$ and using \eqref{eq:indep}--\eqref{eq:event2},
\begin{align*}
&
\mathbb P(I_1^{(n)}(t)=1,I_2^{(n)}(t)=1\mid Y)\\
&=
\mathbb P\!\left(R_1>\alpha_n(x_n-h_n(Y))\mid Y\right)
\mathbb P\!\left(R_2>\alpha_n(x_n+h_n(Y))\mid Y\right)
\nonumber\\
&=
q_n\!\left(\alpha_n(x_n-h_n(Y))\right)\,
q_n\!\left(\alpha_n(x_n+h_n(Y))\right).
\end{align*}
Taking expectation,
\begin{equation}
\label{eq:joint_as_E}
\mathbb P(I_1^{(n)}(t)=1,I_2^{(n)}(t)=1)
=
\mathbb E\!\left[
q_n\!\left(\alpha_n(x_n-h_n(Y))\right)\,
q_n\!\left(\alpha_n(x_n+h_n(Y))\right)
\right].
\end{equation}

Similarly,
\begin{equation}
\label{eq:pn_as_E}
p_n(t)=\mathbb P(I_1^{(n)}(t)=1)
=
\mathbb E\!\left[q_n\!\left(\alpha_n(x_n-h_n(Y))\right)\right].
\end{equation}

Therefore, everything reduces to obtaining sufficiently accurate uniform asymptotics
for $q_n(\alpha_n(x_n\pm h_n(Y)))$ when $h_n(Y)=O(n^{-1/2})$.

%%%%%%%%%%%%%%%%%%%%%%%%%%%%%%%%%%%%%%%%%%%%%%%%%%%%%%%%%%%%%%%%%%%%
Let
\begin{equation}
\label{eq:AB}
A_n(y):=q_n\!\left(\alpha_n\left(x_n-\frac{y}{\sqrt{n-1}}\right)\right),
\quad
B_n(y):=q_n\!\left(\alpha_n\left(x_n+\frac{y}{\sqrt{n-1}}\right)\right).
\end{equation}

%%%%%%%%%%%%%%%%%%%%%%%%%%%%%%%%%%%%%%%%%%%%%%%%%%%%%%%%%%%%%%%%%%%%

%%%%%%%%%%%%%%%%%%%%%%%%%%%%%%%%%%%%%%%%%%%%%%%%%%%%%%%%%%%%%%%%%%%%%%%%%

%%%%%%%%%%%%%%%%%%%%%%%%%%%%%%%%%%%%%%%%%%%%%%%%%%%%%%%%%%%%%%%%%%%%

From \eqref{eq:joint_as_E}, \eqref{eq:pn_as_E} and \eqref{eq:AB}, we obtain 
\begin{align}
\mathbb P(I_1^{(n)}(t)=1,I_2^{(n)}(t)=1)
&=
\mathbb E\!\left[A_n(Y)\,B_n(Y)\right],\,\,\,
p_n(t)
=
\mathbb E[A_n(Y)].
\label{eq:joint_AB}
\end{align}
Define 
\begin{equation}
\label{eq:MDelta}
M_n(Y):=\frac{A_n(Y)+B_n(Y)}{2},
\quad
\Delta_n(Y):=\frac{A_n(Y)-B_n(Y)}{2}.
\end{equation}
Then
$$
A_n(Y)=M_n(Y)+\Delta_n(Y),
\quad
B_n(Y)=M_n(Y)-\Delta_n(Y).
$$
Recall that
$
Y=\frac{X_{12}-\mu}{\sigma},\,\,\,\,\mu=\tfrac12 .
$
Since the distribution of $X_{12}$ is symmetric about $1/2$,
we have
$
Y \stackrel{d}{=} -Y.
$

Moreover, by definition,
\[
\Delta_n(y)
=
\frac12\left[
q_n\!\left(\alpha_n\left(x_n-\frac{y}{\sqrt{n-1}}\right)\right)
-
q_n\!\left(\alpha_n\left(x_n+\frac{y}{\sqrt{n-1}}\right)\right)
\right],
\]
which implies that $\Delta_n(\cdot)$ is an odd function, i.e.,
\[
\Delta_n(-y)=-\Delta_n(y).
\]

Therefore, 
$$
\mathbb E[\Delta_n(Y)]
=
\mathbb E[\Delta_n(-Y)]
=
-\mathbb E[\Delta_n(Y)],
$$
and hence
$$
\mathbb E[\Delta_n(Y)]=0.
$$

Hence
\begin{equation}
\label{eq:AB_decomp}
A_n(Y)B_n(Y)=M_n(Y)^2-\Delta_n(Y)^2.
\end{equation}

\medskip

By \eqref{eq:joint_AB} and \eqref{eq:AB_decomp},
\begin{align}
\mathrm{Cov}(I_1^{(n)}(t),I_2^{(n)}(t))
&=
\mathbb E[A_n(Y)B_n(Y)]-\bigl(\mathbb E[A_n(Y)]\bigr)^2
\nonumber\\
&=
\Bigl(\mathbb E[M_n(Y)^2]-\bigl(\mathbb E[M_n(Y)]\bigr)^2\Bigr)
-\mathbb E[\Delta_n(Y)^2]
\nonumber\\
&=
\mathrm{Var}(M_n(Y))-\mathbb E[\Delta_n(Y)^2].
\label{eq:cov_var_minus}
\end{align}

\medskip

By \eqref{eq:MDelta},
\[
2\Delta_n(Y)=A_n(Y)-B_n(Y)
=
q_n\!\left(\alpha_n\left(x_n-\frac{Y}{\sqrt{n-1}}\right)\right)
-
q_n\!\left(\alpha_n\left(x_n+\frac{Y}{\sqrt{n-1}}\right)\right).
\]

Since $Y$ is bounded, $|Y|\le K$ a.s., hence $|\Delta_n|=|Y|/\sqrt{n-1}\le K/\sqrt n$.
%Since $Y$ is bounded, $|Y|/\sqrt{n-1}$ lies in the uniform window \eqref{eq:un_window} for all large $n$ (after enlarging $C(t)$ if needed), hence the error term remains deterministic (depends only on $t$ and $n$).

Thus Lemma~\ref{lem:lemma2_uq}(b) applies after choosing its window constant $C(t)\ge K$:

\begin{equation*}
2\Delta_n(Y)
=
2\cdot \frac{Y}{\sqrt{n-1}}\,
\varphi(x_n)
\left(
1+O\!\left(\frac{(\log\log n)^2}{\log n}\right)
\right).
\end{equation*}
Therefore
\begin{equation*}
\Delta_n(Y)
=
\frac{Y}{\sqrt{n-1}}\,
\varphi(x_n)
\left(
1+O\!\left(\frac{(\log\log n)^2}{\log n}\right)
\right),
\end{equation*}

and
$$
\Delta_n(Y)^2
=
\frac{Y^2}{n-1}\,\varphi(x_n)^2
\left(
1+O\!\left(\frac{(\log\log n)^2}{\log n}\right)
\right),
$$
where we used that $(1+O(\varepsilon_n))^2=1+O(\varepsilon_n)$.
Taking expectations and using $\mathbb E[Y^2]=1$,
\begin{equation}
\label{eq:EDelta2}
\mathbb E[\Delta_n(Y)^2]
=
\frac{\varphi(x_n)^2}{n-1}
\left(
1+O\!\left(\frac{(\log\log n)^2}{\log n}\right)
\right).
\end{equation}

\medskip

From \eqref{eq:phi_xnt_asymp},
$$
\varphi(x_n)
=
\frac{e^{-t}\sqrt{2\log n}}{n}
\left(
1+O\!\left(\frac{(\log\log n)^2}{\log n}\right)
\right),
$$

so
\begin{equation}
\label{eq:phi_square}
\varphi(x_n)^2
=
\frac{2e^{-2t}\log n}{n^2}
\left(
1+O\!\left(\frac{(\log\log n)^2}{\log n}\right)
\right).
\end{equation}
Substitute \eqref{eq:phi_square} into \eqref{eq:EDelta2}:
\begin{equation}
\label{eq:EDelta2_final}
\mathbb E[\Delta_n(Y)^2]
=
\frac{2e^{-2t}\log n}{n^3}
\left(
1+O\!\left(\frac{(\log\log n)^2}{\log n}\right)
\right).
\end{equation}

\medskip

By Lemma~\ref{lem:lemma3_var_bound},
\begin{equation}
\label{eq:VarMn_bound}
\mathrm{Var}(M_n(Y))
=
O\!\left(\frac{(\log n)^2}{n^4}\right).
\end{equation}

\medskip

Insert \eqref{eq:EDelta2_final} and \eqref{eq:VarMn_bound} into \eqref{eq:cov_var_minus}:
\begin{align*}
\mathrm{Cov}(I_1^{(n)}(t),I_2^{(n)}(t))
&=
\mathrm{Var}(M_n(Y))-\mathbb E[\Delta_n(Y)^2]\\
&=
-\frac{2e^{-2t}\log n}{n^3}
\left(
1+O\!\left(\frac{(\log\log n)^2}{\log n}\right)
\right)
+
O\!\left(\frac{(\log n)^2}{n^4}\right).
\end{align*}
Since $O((\log n)^2/n^4)=\frac{2e^{-2t}\log n}{n^3}\cdot O(\log n/n)$,
and $\log n/n=o((\log\log n)^2/\log n)$, we obtain,
\[
\mathrm{Cov}\!\left(I_1^{(n)}(t), I_2^{(n)}(t)\right)
=
-\frac{2\,e^{-2t}\log n}{n^3}
\left(
1+O\!\left(\frac{(\log\log n)^2}{\log n}\right)
\right),
\quad n\to\infty.
\]
\end{proof}

\section*{Acknowledgements}
I thank Boris Alemi for executing the MATLAB program on the department’s high-performance computer.
This research is supported in part by BSF grant 2020063.

{}

\end{document}